\documentclass{amsart}
\usepackage{amssymb}
\usepackage{graphicx}
\usepackage{amscd}
\usepackage{tikz}
\newtheorem{lemma}{Lemma}[section]
\newtheorem{theorem}[lemma]{Theorem}
\newtheorem{proposition}[lemma]{Proposition}
\newtheorem{corollary}[lemma]{Corollary}

\theoremstyle{definition}
\newtheorem{definition}[lemma]{Definition}

\theoremstyle{remark}
\newtheorem{example}[lemma]{Example}
\newtheorem{examples}[lemma]{Examples}
\newtheorem{remark}[lemma]{Remark}

\def\SN{\mathbb N}    %%% N
\def\SZ{\mathbb Z}    %%% Z
\def\SQ{\mathbb Q}    %%% Q
\def\SP{\mathbb P}    %%% R
\def\mm{\mathfrak m}    %%% m
\def\pp{\mathfrak p}    %%% p
\def\qq{\mathfrak q}    %%% q
\def\PP{\mathfrak P}

\def\II{\mathfrak I} 
\def\JJ{\mathfrak J}       %%% m

\def\nn{\mathfrak n}    %%% n

\def\int{\mbox{\rm{Int}}}             %%% Int

\def\Cc{\mathcal{C}}

\def\Aa{\mathcal{A}}

\def\Ss{\mathcal{S}}
\def\Oo{\mathcal{O}}

\def\Pp{\mathcal{P}}

\def\Ii{\mathcal{I}}

\def\Ss{\mathcal{S}}

\def\Gal{\mathrm{Gal}}
\def\Max{\mathrm{Max}}

\def\be{\begin{equation}}
\def\ee{\end{equation}}

\begin{document}

\title[P\'olya groups of galoisian extensions]{From P\'olya fields to P\'olya groups \\(I) Galoisian extensions}

\author{Jean-Luc Chabert}
\date{\today}

%%%%%%%%%%%%%%%%%%%%%%%%%%%%%%%%%%%%%%%%
\maketitle

%%%%%%%%%%%%%%%%%%%%%%

\begin{abstract}
The P\'olya group of a number field $K$ is the subgroup of the class group of $K$ generated by the classes of the products of the maximal ideals with same norm. A P\'olya field is a number field whose P\'olya group is trivial. Our purpose is to start with known assertions about P\'olya fields to find results concerning P\'olya groups. This first paper inspired by Zantema's results focuses on the P\'olya group of the compositum of two galoisian extensions of $\SQ$. 
\end{abstract}

\section{Introduction}

\subsection{P\'olya fields}$\empty$

\smallskip

The notion of P\'olya fields goes back to 1919 with two papers by P\'olya~\cite{bib:polya} and Ostrowski~\cite{bib:ostrowski} where they undertook a study of the integer-valued polynomials in number fields. For a fixed number field $K$ with ring of integers $\Oo_K$, they were interested in the fact that the $\Oo_K$-module $\int(\Oo_K)=\{f(X)\in K[X]\mid f(\Oo_K)\subseteq \Oo_K\}$ admits a regular basis, that is, a basis with one and only one polynomial of each degree. They proved that such a basis exists if and only if, for each integer $q$ which is the norm of a maximal ideal of $\Oo_K$, the product of all maximal ideals of $\Oo_K$ with norm $q$ is a principal ideal.

Let us introduce a notation: for each positive integer $q$, let 
\be \Pi_q(K)=\prod_{\mm\in\Max(\Oo_K),\,\vert \Oo_K/\mm\, \vert =\, q}\,\mm\,.\ee
If $q$ is not the norm of a maximal ideal of $\Oo_K$, then by convention $\Pi_q(K)=\Oo_K$. 
A formal definition was introduced only in 1982 by Zantema~\cite{bib:zantema}. Because of the characterization previously given, it could be stated in the following way:

\begin{definition}
A {\em P\'olya field} is a number field $K$ such that the following equivalent assertions are satisfied:
\begin{enumerate}
\item the $\Oo_K$-module $\int(\Oo_K)$ admits a regular basis,
\item all the ideals $\Pi_q(K)$ are principal.
\end{enumerate}
\end{definition}

Between 1919 and 1982, nothing was done about P\'olya fields. Zantema's paper is a great contribution to the subject which from our point of view is not sufficiently known. After 1982, one can find several papers on the subject, most often they are studies of particular cases: Spickermann~\cite{bib:spicker} in 1986,  Leriche (\cite{bib:leriche1}, \cite{bib:leriche2}, \cite{bib:leriche3}) in 2011--2014, Heidaryan and Rajaei (\cite{bib:BR},  \cite{bib:HR1}, \cite{bib:HR2}) in 2014--2017.

\begin{examples} The following number fields are P\'olya fields:

-- the fields $K$ with class number $h_K=1$,

-- the cyclotomic fields~\cite[Proposition 2.6]{bib:zantema},

-- the abelian extensions of $\SQ$ with only one ramified prime~\cite[Proposition 2.5]{bib:zantema},

-- the Hilbert class fields~\cite[Corollary 3.2]{bib:leriche3}.
\end{examples}

\smallskip

The notion of P\'olya field introduced for number fields has been extended to functions fields by Van Der Linden~\cite{bib:vanderlinden} in 1988. His results were then generalized by Adam~\cite{bib:david1} in 2008. But we will not speak here about the case of positive characteristic.

%%%%%%%%%%%%%%%%%%%%%%%%%%%%%%%%%
\subsection{P\'olya groups}$\empty$

\smallskip

The notion of P\'olya-Ostrowski group was introduced later, in 1997, in a monography about integer-valued polynomials `as  a measure of the obstruction' for K to be a P\'olya field \cite[\S II.4]{bib:CC}. For simplicity, we will call it just P\'olya group.

\begin{definition}
The {\em P\'olya group}  of a number field $K$ is the subgroup $\Pp o(K)$ of the class group $\Cc l(K)$ generated by the classes of the ideals $\Pi_q(K)$.	
\end{definition}

The P\'olya group of $K$ is also the subgroup of $\Cc l(K)$ generated by the classes of the factorials ideals of $K$ as defined by Bhargava~\cite{bib:bhargava1998}, but here we will not use this property.

Of course, the field $K$ is a P\'olya field if and only if its P\'olya group is trivial. So that, every result about P\'olya groups has a corollary which is a result about P\'olya fields. For instance, the first important result about P\'olya groups is due to Hilbert~\cite[Prop. 105--106]{bib:hilbert} who studied the ambiguous ideals since, for a galoisian number field $K,$ the P\'olya group of $K$ is nothing else than the group formed by the classes of ambiguous ideals:

\begin{proposition}[Hilbert, 1897]\label{th:06}
Let $K$ be a quadratic number field. If $s_K$ denotes the number of ramified prime numbers in the extension $K/\SQ$, then
\be \vert \Pp o(K)\vert = \left\{ \begin{array}{cl}2^{s_K-2}&\textrm{if } K \textrm{ is real and } N_{K/\SQ}(\Oo_K^\times)=\{1\}\\ 2^{s_K-1} & \textrm{else.}\end{array}\right.	 \ee  
\end{proposition}

It is then easy to describe the quadratic number fields which are P\'olya fields (cf. \cite[Example 3.3]{bib:zantema} ou \cite[Cor. II.4.5]{bib:CC}). 

\medskip

Since, as previously said, if $K$ is a galoisian number field, the P\'olya group of $K$ is the group of strongly ambiguous classes, we find in the literature some results about P\'olya groups before their introduction. For instance, if $K$ is a quadratic number field, Mollin~\cite[1993]{bib:mollin1993} gave a sufficient condition for $\Cc l(K)$ to be generated by the classes of ramified primes, that is, such that $\Pp o(K)=\Cc l(K)$. In Section~\ref{sec:4} we recall some classical results about the order of this group of strongly ambiguous classes.
About P\'olya groups we may also cite more recent papers: \cite{bib:AC},\cite{bib:chabert2003}, \cite{bib:elliott}, \cite{bib:taous}, \cite{bib:TZ}, \cite{bib:zek}. 
%%%%%%%%%%%%%%%%%%%%%%%%%%%%%%%%%
\subsection{The aim of this paper}$\empty$

\smallskip

Here we are going to make a reverse lecture of this fact that an assertion about P\'olya groups implies, as a particular case, a result about P\'olya fields. We make the conjecture that an assertion about P\'olya fields is the evidence of a statement about P\'olya groups, statement that we have to formulate as a conjecture, and then, to try to prove. For instance, consider the following proposition due to Ostrowski:

\begin{proposition}\cite{bib:ostrowski}\label{th:12}
If the extension $K/\SQ$ is galoisian, then $K$ is a P\'olya field if and only if, for every non ramified prime number $p$, the ideal $\Pi_{p^{f_{p(K/\SQ)}}}(K)$ is principal.	
\end{proposition}

This is an easy consequence of the fact that, if the extension $K/\SQ$ is galoisian, if $p$ is not ramified, and if $q=p^{f_{p(K/\SQ)}}$, then $\Pi_q(K)=p\Oo_K$ is principal. Thus, clearly:

\begin{corollary}
If the extension $K/\SQ$ is galoisian, then the P\'olya group of $K$ is generated by the classes of the ideals $\Pi_{p^{f_p(K/\SQ)}}(K)$ where $p$ is ramified in the extension $K/\SQ$.
\end{corollary}

We are going to take advantage of the results of Zantema's paper on P\'olya fields to obtain results on P\'olya groups. In this first study we restrict ourselves to the easier case, that is the galoisian case. Here is Zantema's theorem that we want to transpose. It concerns the compositum of two galoisian extensions of $\SQ$.

\begin{theorem}\cite[Thm 3.4]{bib:zantema}\label{th:21}
Let $K_1$ and $K_2$ be finite Galois extensions	of $\SQ$, $L=K_1\cdot K_2$, $K=K_1\cap K_2$. For a prime $p,$ let $e_i(p)$ be the ramification index of $p$ in $K_i/K$, $i=1,2$. Then
\begin{enumerate}
\item[(a)] If $\gcd(e_1(p),e_2(p))=1$ for all primes $p$ and if $K_1$ and $K_2$ are P\'olya fields, then $L$ is a P\'olya field.
\item[(b)] If either $[K_1:K], [K_2:K], [K:\SQ]$ are pairwise relatively prime, or $K$ is a P\'olya field and $\gcd([K_1:K],[K_2:K])=1$, then $K_1$ and $K_2$ are P\'olya fields if $L$ is a P\'olya field.
\end{enumerate}
\end{theorem}

The transposition of this theorem in terms of P\'olya groups will be done with Proposition~\ref{th:213} and Theorems~\ref{th:24} and \ref{th:241}.

%%%%%%%%%%%%%%%%%%%%%%%%%%%%%

\section{Galoisian extensions of $\SQ$}\label{sec:2}

\noindent{\bf Notation}. For every number field $K$, we denote by $\Oo_K$ its ring of integers, $\Oo_K^\times$ its group of units, $\Ii_K$ the group a nonzero fractional ideals, $\Pp_K$ the subgroup of nonzero principal fractional ideals, $\Cc l(K)=\Ii_K/\Pp_K$ its class group, and $h_K=\vert \Cc l(K)\vert$ its class number.

Recall that, for every $q\in\SN^*$, $\Pi_q(K)$ denotes the product of the maximal ideals of $\Oo_K$ with norm $q$ and that $\Pp o(K)$ is the subgroup of $\Cc l(K)$ generated by the classes of these $\Pi_q(K)$'s.

If $K$ is a galoisian extension of $\SQ$, for simplicity we will write $\Pi_{p^*}(K)$ instead of $\Pi_{p^{f_{p(K/\SQ)}}}(K)$, and if $G$ denotes the Galois group $\Gal(K/\SQ)$, $\Ii_K^G$ will denote the subgroup of $\Ii_K$ formed by the ambiguous ideals of $K$. 

\smallskip

Note that, in the galoisian case, the group of ambiguous ideals $\Ii_K^G$ is generated by the ideals $\Pi_{p^*}(K)$, in fact, this is the free abelian group with basis formed by the $\Pi_{p^*}(K)$ that are distinct from $\Oo_K$. As a consequence, the image of $\Ii_K^G$ in the class group $\Cc l(K)$ is the P\'olya group of $K:$
\be \Pp o(K) = \Ii_K^G/\Pp_K^G\,.\ee

%%%%%%%%%%%%%%%%%%%%%%

In order to establish links between different P\'olya groups, we have to recall several homomorphisms of groups (cf.~\cite[I, \S 5]{bib:serre}). 
If L/K is a finite extension, there is a natural injective morphism
\be j_K^L:\II\in\Ii_K\mapsto \II\Oo_L\in\Ii_L \ee
which induces a (non necessarily injective) morphism:
\be \varepsilon_K^L:\overline{\II}\in\Cc l(K)\mapsto\overline{\II\Oo_L}\in\Cc l(L).\ee
On the other hand, the norm morphism 
$$ N_L^K:\Ii_L\to\Ii_K,$$
which is determined by its values on the maximal ideals $\nn$ of $\Oo_L:$
\be N_L^K(\nn)=\mm^{[\Oo_L/\nn \,:\, \Oo_K/\mm]} \;\textrm{ where } \;\mm=\nn\cap \Oo_K,\ee 
induces a morphism:
\be \nu_L^K:\overline{\JJ}\in\Cc l(L)\mapsto\overline{N_L^K(\JJ)}\in\Cc l(K)\,.\ee 
And, for every $\II\in\Ii_K$, one has:
\be (N_L^K\circ j_K^L)(\II)=\II^{[L:K]}.\ee

As  noticed in \cite{bib:chabert2003}, if $K$ and $L$ are galoisian extensions of $\SQ$ with Galois groups $H$ and $G$ respectively, then the subgroups $\Ii_K^H$ and $\Ii_L^G$ on the one hand, and the subgroups $\Pp o(K)$ and $\Pp o(L)$ on the other hand, behave well with respect to the previous morphisms:

\begin{lemma} If $K\subseteq L$ are two galoisian extensions of $\SQ$ with respective Galois groups $H$ and $G$, then:
$$j_K^L(\Ii_K^H)\subseteq \Ii_L^G\quad and\quad \varepsilon_K^L(\Pp o(K))\subseteq \Pp o(L)\,,$$
$$N_L^K(\Ii_L^G)\subseteq\Ii_K^H\quad and\quad \nu_l^K(\Pp o(L))\subseteq \Pp o(K)\,.$$
\end{lemma}

This is no more true if one of the fields is not a galoisian extension of $\SQ$ (cf. \cite[Ex. 2]{bib:leriche1}). We begin with a lemma which is a `local version' of the behaviour of the P\'olya groups.

\begin{lemma}\label{le:24}
Let $K_0\subseteq K_1\subseteq L$ be three galoisian extensions of  $\SQ$ and let $p\in\SP$ be such that:
\begin{enumerate}
\item $\Pi_{p^*}(K_0)$ is principal,
\item $e_p(K_1/K_0)$ et $[L:K_1]$ are relatively prime.
\end{enumerate}
Then, the morphism $\nu_L^{K_1}\circ \varepsilon_{K_1}^L:\Pp o(K_1)\to\Pp o(L)\to\Pp o(K_1)$ induces an automorphism of the group $\langle \,\overline{\Pi_{p^*}(K_1)}\,\rangle $. 
In particular, $\langle\,\overline{\Pi_{p^*}(L)}\, \rangle$ contains a subgroup isomorphic to $\langle\,\overline{\Pi_{p^*}(K_1)}\,\rangle $. 
\end{lemma}

\begin{proof}
Let $\Pi=\Pi_{p^*}(K_1)$. The equality $\Pi_{p^*}(K_0)\Oo_K=\Pi^{e_p(K_1/K_0)}$ 	
implies with hypothesis (1) that the order of the class $\overline{\Pi}$ of $\Pi$ divides $e_p(K_1/K_0)$. It follows from hypothesis (2) that this order is prime to $[L:K_1]$.
The formula $\nu_L^K(\varepsilon_K^L(\overline{\Pi}))=\overline{\Pi}^{[L:K_1]}$ shows then that the morphism $\nu_L^{K_1}\circ\varepsilon_{K_1}^L$ induces an automorphism of the group generated by $\overline{\Pi_{p^*}(K_1)}$. The last assertion is obvious.
\end{proof}

\noindent By `globalization' Lemma~\ref{le:24} leads to:

\begin{proposition}\label{cor:26}
Let $K_0\subseteq K_1\subseteq L$ be three galoisian extensions of $\SQ$ such that: 
\begin{enumerate}
\item $K_0$ is a P\'olya field,
\item $([L:K_1],[K_1:K_0])=1$.
\end{enumerate}
Then, $\Pp o(L)$ contains a subgroup which is isomorphic to $\Pp o(K_1)$. More precisely,
$$\Pp o(L)\simeq \Pp o(K_1)\oplus \left( \, \Pp o(L)\,/\,\varepsilon_{K_1}^L(\Pp o(K_1)) \,\right) $$
\end{proposition}

\begin{proof}  Let $n_0=[K_0:\SQ]$, $n_1=[K_1:K_0]$ et $n_2=[L:K_1]$.

\bigskip

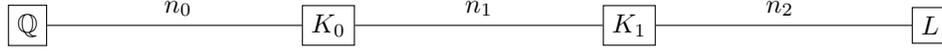
\begin{tikzpicture}

\node[draw] (Q) at (0,0) {$\SQ$};
\node[draw] (K0) at (4,0) {$K_0$};
\node[draw] (K1) at (8,0) {$K_1$};
\node[draw] (L) at (12,0) {$L$};

\node[above] (n0) at (2,0) {$n_0$};
\node[above] (n1) at (6,0) {$n_1$};
\node[above] (n2) at (10,0) {$n_2$};

\draw (Q) edge[] (K0);
\draw (K0) edge[] (K1);
\draw (K1) edge[] (L);

\end{tikzpicture}

\bigskip

The group $\Pp o(K_1)$ is generated by the classes of the $\Pi_{p^*}(K_1)$'s whose orders divide $e_p(K_1/K_0)$, which itself divides $n_1$. Thus, the order of every element of the abelian group $\Pp o(K_1)$ divides $n_1$. Consequently the restriction of the morphism $\nu_L^{K_1}\circ \varepsilon_{K_1}^L$ to $\Pp o(K_1),$ which corresponds to the  elevation to the power $n_2,$ is an automorphism of $\Pp o(K_1)$ since $(n_1,n_2)=1$. Thus, the subgroup $\varepsilon_{K_1}^L(\Pp o(K_1))$ of $\Pp o(L)$ is isomorphic to $\Pp o(K_1)$.

On the other hand, the group $\Pp o(L)$ is generated by the classes of the $\Pi_{p^*}(L)$'s. But, the subgroup $\varepsilon_{K_1}^L(\Pp o(K_1))$ contains $\Pi_{p^*}(K_1)\Oo_L=\Pi_{p^*}(L)^{e_p(L/K_1)}$. Thus, mod $\varepsilon_{K_1}^L(\Pp o(K_1))$, the order of $\overline{\Pi_{p^*}(L)}$ divides $e_p(L/K_1)$ which itself divides $n_2$. Consequently, the order of every element of $\Pp o(L)$ mod $\varepsilon_{K_1}^L(\Pp o(K_1))$ divides $n_2$. 

As all the orders of the elements of $\Pp o(L)$ divide $n_1n_2$, one has:
$$\Pp o(L)= G_{n_1}\oplus G_{n_2}\quad\textrm{where} \quad G_{n_i}=\{\alpha\in\Pp o(L)\mid \alpha^{n_i}=1\}\,.$$ Clearly, 

\quad\quad\quad $G_{n_1}=\varepsilon_{K_1}^L(\Pp o(K_1))\simeq \Pp o(K_1) \quad\textrm{ and }\quad G_{_2}\simeq \Pp o(L)\,/\,\varepsilon_{K_1}^L(\Pp o(K_1)) \,.$
\end{proof}

\noindent In particular, if $K_0=\SQ,$ one has:

\begin{corollary}\label{cor:26a}
	Let $K_1\subseteq L$ be two galoisian extensions of $\SQ$ such that one has $([L:K_1],[K_1:\SQ])=1$. Then $\Pp o(L)$ contains a subgroup isomorphic to $\Pp o(K_1)$. In particular, if $L$ is a P\'olya field, $K_1$ is also a P\'olya field.
\end{corollary}

\begin{example}
 Let $K$ be a non-galoisian cubic number field. Denote by $L$ the galoisian closure of $K$ and by $K'$ the unique quadratic number field such that $L=KK'$. Then,  $\Pp o(K')\subseteq \Pp o(L).$	 
\end{example}

%%%%%%%%%%%%%%%%%%%%%%%

\section{Compositum of galoisian extensions}\label{sec:3}

First let us recall a strong version of Abhyankar's lemma (see for instance~\cite[Prop. III.8.9]{bib:stich} for the function field case): 

\begin{proposition}\cite{bib:EH}\label{th:113}
Let $A$ be a Dedekind domain with fraction field $K$. Let $\pp$ be a maximal ideal of $A$ and assume that the residue field $A/\pp$ is perfect. Let $K_1$ and $K_2$ be two finite extensions of $K$. Denote by $L$ the compositum of $K_1$ and $K_2$ and assume that the extension $L/K$ is separable. Let $\qq$ be a prime ideal of  $L$ such that $\qq\cap K=\pp$. Let $\pp_1=\qq\cap K_1$, $\pp_2=\qq\cap K_2$, and $p\SZ=\qq\cap \SZ$. If $p$ does not divide one of the ramification indices $e(\pp_i/\pp)$ $(i=1,2)$, then $$e(\qq/\pp)=\mathrm{lcm}\,\{e(\pp_1/\pp),e(\pp_2/\pp)\}.$$
\end{proposition}

\noindent{\bf Hypotheses and notation}. Let $K_1$ and $K_2$ be two galoisian extensions of $\SQ$. 

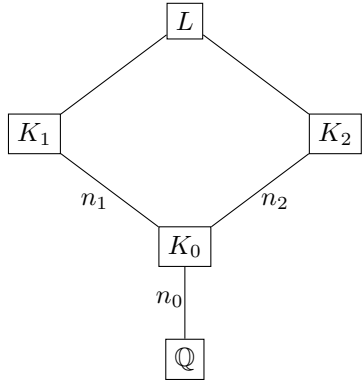
\medskip

\begin{tikzpicture}

\node[draw] (Q) at (5,0) {$\SQ$};
\node[draw] (K0) at (5,1.5) {$K_0$};
\node[draw] (K1) at (3,3) {$K_1$};
\node[draw] (K2) at (7,3) {$K_2$};
\node[draw] (L) at (5,4.5) {$L$};

\node (n0n1n2) at (12,2) {$K_0=K_1\cap K_2$};
\node (n1n2) at (12,4) {$L=K_1K_2$};

\node (n1) at (3.8,2.1) {$n_1$};
\node (n2) at (6.2,2.1) {$n_2$};
\node (n0) at (4.8,.8) {$n_0$};

%\node[draw , shape=circle] (G) at (-2,7) {$G$};
%\node[draw , shape=circle] (T) at (6,11) {$<T>$};
%\node[draw , shape=circle] (H2) at (8,12) {$H_2$};
%\node[draw , shape=circle] (TH2) at (9,8) {$<T>/H_2$};
%\node[draw , shape=circle] (HH2) at (9,5) {$H/H_2$};
%\node[draw , shape=circle] (GH2) at (9,2) {$G/H_2$};

\draw (Q) edge[] (K0);
\draw (K0) edge[] (K1);
\draw (K0) edge[] (K2);
\draw (K1) edge[] (L);
\draw (K2) edge[] (L);
%\draw (L) edge[] (NK);
%\draw (L) edge[] (NH);

%\draw (5,14) arc (90:270:7);
%\draw (5,14) arc (90:270:5);
%\draw (5,4) arc (250:360:5);
%\draw (5,0) arc (270:360:7);

\end{tikzpicture}

\begin{lemma}\label{th:c'}
Let $p$ be a prime number such that $p$ does not divide at least one of the ramification indices $e_p(K_i/K_0)$. Then, in the group $I_{\Oo_L}$, one has the following equality of subgroups:
\be\label{eq:1} \langle\,\Pi_{p^*}(K_1)\Oo_L,\Pi_{p^*}(K_2)\Oo_L\,\rangle \; = \; \langle\,\Pi_{p^*}(L)\,\rangle\,,\ee
and hence, in the group $\Cc l(L):$ 
\be \varepsilon_{K_1}^L(\langle\, \overline{\Pi_{p^*}(K_1)}\,\rangle) + \varepsilon_{K_2}^L(\langle\,\overline{\Pi_{p^*}(K_2)}\,\rangle)= \,\langle\,\overline{\Pi_{p^*}(L)}\,\rangle\,.  \ee
In particular, the ideal $\Pi_{p^*}(L)$ is principal if and only if the ideals $\Pi_{p^*}(K_1)\Oo_L$ and $\Pi_{p^*}(K_2)\Oo_L$ are themselves principal.
\end{lemma}

\begin{proof}
Clearly, $L$ and $K_0$ are galoisian extensions of $\SQ$. Note that $K_0$ could be replaced by any field contained in $K_1\cap K_2$ provided that $K_0$ is a galoisian extension of $\SQ$.  
  Let  $e_i=e_p(K_i/K_0) \; (i=1,2)$ and $m=\mathrm{lcm}(e_1,e_2)$. Following Abhyankar's lemma, $e_p(L/K)=m$. Let 
$$\pi=\Pi_{p^*}(K_0)\,,\;\Pi_i=\Pi_{p^*}(K_i)\;(i=1,2)\textrm{ and }\Pi=\Pi_{p^*}(L)\,.$$ 
Then, 
\be\label{eq:2}\pi\Oo_{K_i}=\Pi_i^{e_i},\;\pi\Oo_L=\Pi^{m},\;\Pi_i\Oo_L=\Pi^{\frac{m}{e_i}}\;(i=1,2).\ee
Let $u_1$ and $u_2\in\SZ$ be such that $u_1\times \frac{m}{e_1}+u_2\times \frac{m}{e_2}=1 $. Then, one has:
\be\label{eq:3} \Pi_1^{u_1}\Pi_2^{u_2}\Oo_L =\Pi^{u_1\frac{m}{e_1}+u_2\frac{m}{e_2}}=\Pi  .\ee
Clearly, egality (\ref{eq:1}) follows from equalities (\ref{eq:2}) and (\ref{eq:3}).
\end{proof}

\begin{lemma}\label{th:117bis}
Under the hypothesis of lemma \ref{th:c'}, 	the equality
$$j_{K_1}^L(\langle\,\Pi_{p^*}(K_1)\,\rangle)\cap j_{K_2}^L(\langle\,\Pi_{p^*}(K_2)\,\rangle) = j_K^L(\langle\,\Pi_{p^*}(K_0)\,\rangle)$$
holds if and only if $e(\pp_1/\pp)$ and $e(\pp_2/\pp)$ are coprime
\end{lemma}

\begin{proof}
Clearly $j_{K_0}^L(\Pi_{p^*}(K_0))\subseteq j_{K_1}^L(\Pi_{p^*}(K_1))\cap j_{K_2}^L(\Pi_{p^*}(K_2))$. Do we have an equality?
The ideals $I$ belonging to this intersection are of the form $I=\Pi_i^{k_i}\Oo_L$ where $k_i\in\SZ$ ($i=1,2$), thus of the form $I=\Pi^{k_i\frac{m}{e_i}}$ ($i=1,2$). Consequently, $I=\Pi^\alpha$ where $\alpha$ is divisible by $\frac{m}{e_1}$ and $\frac{m}{e_2},$ that is, multiple of 
$\mathrm{lcm}(\frac{m}{e_1},\frac{m}{e_2})$, namely, multiple of $\frac{m}{d}$ where $d=\gcd(e_1,e_2)$. Letting $k_i=\frac{e_i}{d}$, one has $I=\Pi^{\frac{m}{d}}$, while $\pi\Oo_L=\Pi^m$.  
\end{proof}

\noindent By `globalization', Lemmas~\ref{th:c'} and \ref{th:117bis} lead to both following propositions.

\begin{proposition}\label{th:117b}
Let $K_1$ and $K_2$ be two finite galoisian extensions of $\SQ$. Let  $K_0=K_1\cap K_2$, $L=K_1K_2$, $G=\Gal(L/\SQ)$ and $G_i=\Gal(K_i/\SQ)$ for $0\leq i\leq 2$. If $([K_1:K_0],[K_2:K_0])=1$, then,
$$j_{K_1}^L(\Ii_{K_1}^{G_1})\cdot j_{K_2}^L(\Ii_{K_2}^{G_2})=\Ii_L^G \;\;\textrm{ and }\;\; j_{K_1}^L(\Ii _{K_1}^{G_1})\cap j_{k_2}^L(\Ii_{K_2}^{G_2})=j_{K_0}^L(\Ii_{K_0}^{G_0})\,.$$
Thus, $\Ii_L^G$ is isomorphic to  the amalgamated sum of $\Ii_{K_1}^{G_1}$ and $\Ii_{K_2}^{G_2}$ above $\Ii_{K_0}^{G_0}:$ 
$$\Ii_L^G\simeq \;\Ii_{K_1}^{G_1}\;*_{\Ii_{K_0}^{G_0}}\;\Ii_{K_2}^{G_2} $$
with respect to the morphisms $j$.
\end{proposition}

\begin{proof}
This is a consequence of the fact that the groups $\Ii_{K_i}^{G_i}$ are free abelian groups with bases formed by the $\Pi_{p^*}(K_i)$'s.
\end{proof}

\begin{proposition}\label{th:213}
Let $K_1$ and $K_2$ be two finite galoisian extensions of $\SQ$. Let $K_0=K_1\cap K_2$ and $L=K_1K_2$. If, for each prime number $p$, at least one extension $K_i/K_0 \; (i=1,2)$ is tamely ramified with respect to  $p$, then
$$\Pp o(L)=\varepsilon_{K_1}^L(\Pp o(K_1)) + \varepsilon_{K_2}^L(\Pp o(K_2))\,. $$
\end{proposition}

\begin{example}
Let $K_1=\SQ(\sqrt{d})$ and $K_2=\SQ(\sqrt{d'})$. If $d\equiv 1\pmod{4}$, then
$$\Pp o(\SQ(\sqrt{d},\sqrt{d'}))=\varepsilon_{K_1}^L(\Pp o(K_1)) + \varepsilon_{K_2}^L(\Pp o(K_2))\,. $$
\end{example}

Proposition~\ref{th:213} implies as a particular case assertion (a) of Zantema's Theorem~\ref{th:21}. Both following theorems describe the group $\Pp o(L)$ by generalizing assertion (b) of Theorem~\ref{th:21}.

\begin{theorem}\label{th:24}
Let $K_1$ and $K_2$ be two finite galoisian extensions of $\SQ$. Let $K_0=K_1\cap K_2$ and  $L=K_1K_2$. Assume that:
\begin{enumerate}
\item[(a)] $K_0$ is a P\'olya field,
\item[(b)] $([K_1:K_0],[K_2:K_0])=1$.
\end{enumerate}
Then, 
$$\Pp o(L)\simeq \Pp o(K_1)\oplus  \Pp o(K_2)\,.$$
More precisely, 
$$\Pp o(L)= \varepsilon_{K_1}^L(\,\Pp o(K_1)\,)\oplus  \varepsilon_{K_2}^L(\,\Pp o(K_2)\,)\,. $$
\end{theorem}

\begin{proof}$\empty$
This is a consequence of Proposition~\ref{th:213} and of the fact that the intersection of the two subgroups $\varepsilon_{K_1}^L(\Pp o(K_1))$ and  $\varepsilon_{K_2}^L(\Pp o(K_2))$ in $\Pp o(L)$ is $\{1\}$. Indeed, the order of every element of $\Pp o(K_i)$ divides $[K_i:K_0]$ ($i=1,2$) and, by hypothesis, these degrees are coprime.
\end{proof}

\begin{example}
Let $K_1=\SQ(i,j)$ and $K_2=\SQ(j,\sqrt[3]{7})$. Then, $\Pp o(K_1)=\{1\}$~\cite[Thm. 5.1]{bib:leriche2} and $\Pp o(K_2)=\Cc l(K_2)\simeq \SZ/3\SZ$~\cite[Ex. 6.8]{bib:leriche2}. As $K_0=\SQ(j)$ is a P\'olya field, we obtain: $\Pp o(\SQ(i,j,\sqrt[3]{7}))\simeq \SZ/3\SZ$.
\end{example}

\begin{corollary}
Assume $K_0$ is a P\'olya field and that $L/K_0$ is a galoisian extension whose Galois group is of the form $\prod_{i\in I}G_i$ where, for $i\not=j$, $(\vert G_i\vert,\vert G_j\vert)=1.$ Then,
$$\Pp o(L)\simeq \oplus_{i\in I}\,\Pp o(K_i)$$
where $K_i$ denotes the subfield of $L$ fixed by all the $G_j$ where $j\not=i$.	
\end{corollary}

As a particular case, we have:

\begin{corollary}\label{cor:22}
Assume that $K_0$ is a P\'olya field and that $L/K_0$ is an abelian extension of degree $n=\prod_{p\in\SP}p^{\nu(p)}$. Then,
$$\Pp o(L)\simeq \oplus_{p\vert n}\Pp o(K_p)$$
where, for each $p$ dividing $n$, $K_p$ denotes the subfield of $L$ of degree $p^{\nu(p)}$ on $K_0$.	
\end{corollary}

As a particular case, we obtain \cite[Zantema, Cor. 3.5]{bib:zantema}. 

\begin{theorem}\label{th:241}
Let $K_1$ and $K_2$ be two finite galoisian extensions of $\SQ$. Let $K_0=K_1\cap K_2$ and $L=K_1K_2$. If the degrees $[K_1:K_0]$, $[K_2:K_0]$ and $[K_0:\SQ]$ are pairwise relatively prime, then
$$\Pp o(L)\simeq \Pp o(K_0) \oplus \left(\,\Pp o(K_1)\,/\,\varepsilon_{K_0}^{K_1}(\Pp o(K_0))\,\right)\oplus  \left(\, \Pp o(K_2)\,/\,\varepsilon_{K_0}^{K_2}(\Pp o(K_0))\,\right)\,. $$
\end{theorem}

\begin{proof}
This is a consequence of Propositions~\ref{cor:26} and \ref{th:213}. Indeed, for $i=1,2$, one has the canonical injective morphisms :
$$\varepsilon_{K_0}^{K_i}: \Pp o(K_0)\to \Pp o(K_i)=\varepsilon_{K_0}^{K_i}(\Pp o(K_0))\oplus W_i \;\textrm{where }\; W_i\simeq \Pp o(K_i)/\varepsilon_{K_0}^{K_i}(\Pp o(K_0))\,.$$
Moreover,
$$\Pp o(L)= \varepsilon_{K_1}^L(\Pp o(K_1)+\varepsilon_{K_2}^L(\Pp o(K_2))\,.$$
Thus,
$$\Pp o(L)=\varepsilon_{K_0}^L(\Pp o(K_0))+\varepsilon_{K_1}^L(W_1)+\varepsilon_{K_2}^L(W_2)  $$
and this sum is direct since the orders of the elements of the three subgroups are pairwise relatively prime.
\end{proof}

%%%%%%%%%%%%%%%%%%%%%%
%%%%%%%%%

\section{The relative P\'olya group of an extension}\label{sec:4}

We introduce the following notation for every finite extension $L/K$ of number fields: 
\be \forall \pp\in\Max(\Oo_K) \; \forall f\in\SN^* \quad \Pi_{\pp^f}(L/K)=\prod_{\PP\in\Max(\Oo_L),\,N_{L/K}(\PP)=\pp^f}\,\PP\,.\ee
This is clearly a generalization of the ideals $\Pi_q(K)$'s since $\Pi_{p^f}(K)=\Pi_{(p\SZ)^f}(K/\SQ)$. Analogously, we generalize the notion of P\'olya group: 

\begin{definition}
The {\em relative P\'olya group} of an extension $L/K$ is the subgroup $\Pp o(L/K)$ of the class group of $L$ generated by the classes of the ideals $\Pi_{\pp^f}(L/K)$.  	
\end{definition}

Clearly, $\Pp o(K)=\Pp o(K/\SQ)$. Here, we are interested in galoisian extensions. 

\smallskip

\noindent{\bf Notation}. In the remaining of this section, {\em the extension $L/K$ is assumed to be galoisian with Galois group $H$}. We may forget the exponent $f$ and write:
$$\Pi_{\pp^*}(L/K)=\Pi_{\pp^{f_{\pp}(L/K)}}(L/K)=\prod_{\PP\in\Max(\Oo_L),\,\PP\vert\pp}\PP\,.$$ 
\noindent For every $\pp\in\Max(\Oo_K)$, one has $\pp\Oo_L=\Pi_{\pp^*}(L/K)^{e_{\pp}(L/K)}$. 
Clearly, the subgroup $\Ii_L^H$ of $\Ii_L$ formed by the invariant ideals is generated by the $\Pi_{\pp^*}(L/K)$'s:
$$\Ii_L^H=\langle\,\Pi_{\pp^*}(L/K)\mid \pp\in\Max(\Oo_K)\,\rangle\,.$$
In fact, $\Ii_L^H$ is the free abelian group with basis formed by the $\Pi_{\pp^*}(L/K)$'s.
Note that the relative P\'olya group $\Pp o(L/K)$ of a galoisian extension $L/K$ is nothing else than the subgroup $\Ii_L^H/\Pp_L^H$ of $\Cc l(L)$ formed by the strongly ambiguous classes with respect to the extension $L/K$.

\smallskip

\noindent Recall that Spickermann generalized the notion of P\'olya field in the following way:

\smallskip

\begin{definition}\cite{bib:spicker}
A galoisian extension of number fields $L/K$ with Galois group $H$ is said to be a {\em galoisian P\'olya extension} if $\Ii_L^H\subseteq \Ii_K\cdot \Pp_L$
\end{definition}

This is a generalization of the definition of  a P\'olya field $L$ only if $L$ is a galoisian extension of $\SQ$. Moreover, with respect to the notion of relative P\'olya group, to say that the extension $L/K$ is a galoisian P\'olya extension means that:
$$\Pp o(L/K)=\varepsilon_K^L(\Cc l(K))\,.$$

If $L/\SQ$ is galoisian, then  $\Pi_{p^*}(L)=\prod_{\pp\vert p}\Pi_{\pp^*}(L/K)$, and hence,
$$\Pp o(L/\SQ)=\Pp o(L)\subseteq \Pp o(L/K)\,.$$

\begin{remark}
There is another notion of P\'olya extension introduced in~\cite[\S\,II.9]{bib:chabert2006} and \cite{bib:leriche1}, that is more directly linked to the integer-valued polynomials and that is completely distinct from the previous definition: in these papers, an extension of number fields $L/K$ is said to be a P\'olya extension if $\varepsilon_K^L(\Pp o(K))=\{1\}$. 	
\end{remark}

We recall now some known results with respect to the group formed by the strongly ambiguous classes.

\begin{proposition}
Let $L/K$ be a galoisian extension with Galois group $H$. Then,
\be\label{eq:13cd} \vert\Pp o(L/K)\vert = h_K \times \frac{\prod_{\pp\in\Max(\Oo_K)}e_{\pp}(L/K)}{\vert H^1(H,\Oo_L^\times)\vert } \ee
\end{proposition}

\begin{proof}
Consider the surjective morphism:
$$\varphi : \prod_{\pp\in\Max(\Oo_K)}\Pi_{\pp^*}(L/K)^{\alpha_{\pp}}\in \Ii_L^H\mapsto \left(\alpha_{\pp}\textrm{ mod }e_{\pp}(L/K)\right)\in\oplus_{\pp\in\Max(\Oo_K)}\;\SZ/e_{\pp(L/K)}\SZ \,.$$  
Clearly, $I=\prod_{\pp\in\Max(\Oo_K)}\Pi_{\pp^*}(L/K)^{\alpha_{\pp}}$ belongs to $\mathrm{Ker}(\varphi)$ if and only if $e_{\pp}(L/K)\mid\alpha_{\pp}$ for each $\pp$, that is, $I=\prod_{\pp\in\Max(\Oo_K)}(\pp\Oo_L)^{\beta_{\pp}}.$ Thus, 
$\textrm{Ker}(\varphi)=j_K^L(\Ii_K)$ and we have the following exact sequence:
\be\label{eq:etoile} 1\to \Ii_K \stackrel{j_K^L} \to  \Ii_L^H\stackrel{\varphi}\to \oplus_{\pp\in\Max(\Oo_K)}\;\SZ/e_{\pp(L/K)}\SZ \to 0 \ee
On the other hand, applying the functor $U\mapsto U^H$ to the exact sequence
$$ 1\to\Oo_L^{\times}\to L^*\to \Pp_L\to 1, $$
leads with Hibert 90 to the following exact sequence:
\be\label{eq:2etoile} 1\to \Pp_K \stackrel{j_K^L} \to \Pp_L^H \stackrel{\delta}\to H^1(H,\Oo_L^\times)\to 1 \,.\ee 
Putting together the exact sequences (\ref{eq:etoile}) and (\ref{eq:2etoile}), we obtain the following commutative diagram:
\minCDarrowwidth16pt
\be\label{CD1}\begin{CD}
1 @>>> \Pp_K @> {j_K^L} >> \Pp_L^H @> {\delta} >> H^1(H,\Oo_L^{\times}) @>>> 1 \\
@ VVV @VVV @VVV @ VV {\tau_{L/K}} V @ VVV\\
1 @>>> \Ii_K @ > {j_K^L} >> \Ii_L^H @ > {\varphi} >>  \oplus_{\pp}\,\SZ/e_{\pp(L/K)}\SZ @ >>> 0
\end{CD}\ee

Note that the morphisme $\tau_{L/K}$ is naturally defined by the commutativity of the diagram. The snake's lemma leads to the following exact sequence:
\be\label{eq:15.} 1\to\textrm{Ker}(\tau_{L/K})\to\Cc l(K) \stackrel{\varepsilon_K^L}\longrightarrow \Pp o(L/K)\to \mathrm{Coker}(\tau_{L/K})\to 1\,.\ee 
As $$\vert \,\mathrm{Coker}(\tau_{L/K})\,\vert = \prod_{\pp}\,e_{\pp}(L/K)\times \vert \,\textrm{Ker}(\tau_{L/K})\,\vert \, /\, \vert H^1(H,\Oo_L^\times)\vert \,,$$ 
Formula (\ref{eq:13cd}) follows from sequence (\ref{eq:15.}).
\end{proof}

\begin{corollary}
Let $L/K$ be a galoisian extension with Galois group $H$ and assume that $h_K=1$. Then, the following sequence of abelian groups is exact: 
\be\label{eq:13ab} 1\to H^1(H,\Oo_L^\times)\stackrel{\tau_{L/K}}\longrightarrow \oplus_{\pp\in\Max(\Oo_K)}\,\SZ/e_{\pp(L/K)}\SZ\to \Pp o(L/K)\to 0\,.\ee
Moreover, $L/K$ is a galoisian P\'olya extension if and only if
$$\vert H^1(H,\Oo_L^\times)\vert = \prod_{\pp\in\Max(\Oo_K)}\,e_{\pp}(L/K)\,.$$

\end{corollary}

\noindent Recall that the morphisme $\tau_{L/K}$ was defined by the commutativity of diagram~(\ref{CD1}). 

\begin{proof}
If $h_K=1$, sequence (\ref{eq:15.}) shows that $\textrm{Ker}(\tau_{L/K})=\{1\}$ and $\textrm{Coker}(\tau_{L/K})=\Ii_L^H/\Pp_L^H=\Pp o(L/K)$ which leads to sequence~(\ref{eq:13ab}). 
Moreover, since $h_K=1$, to say that $L/K$ is a galoisian P\'olya extension means that $\Pp o(L/K)=\{1\}$.
\end{proof}

As a particular case, sequence (\ref{eq:13ab}) leads, when $K=\SQ$, to Zantema's sequence:

\begin{corollary}\cite[p. 163]{bib:zantema}
If $L/\SQ$ is a galoisian extension with Galois group $G$, then the following sequence is exact:
\be\label{eq:8bis} 1\to H^1(G,\Oo_L^{\times})\to\oplus_{p\in\SP}\;\SZ/e_{p(L/\SQ)}\SZ\to \Pp o(L)\to 0.\ee	
\end{corollary}

\begin{corollary}\label{th:44d}
If $L/K$ is a cyclic extension, then $$\vert\Pp o(L/K)\vert= \frac{h_K}{[L:K]}\times\frac{\prod_{\pp\in\Max(\Oo_K)}e_{\pp}(L/K)\times \prod_{\pp\vert\infty}e_{\pp}(L/K)}{(\Oo_K^\times:N_{L/K}(\Oo_L^\times))}.$$
\end{corollary}

\begin{proof}	
Assume that the cyclic group $H=\Gal(L/K)$ is generated by $\sigma$. One knows that ~\cite[IV.3.7]{bib:neukirch}:
$$H^1(H,\Oo_L^{\times})\simeq H^{-1}(H,\Oo_L^{\times})= 
\frac{\{a\in\Oo_L^{\times}\mid N_{L/K}(a)=1\}}
{\{\sigma(a)/a\mid a\in\Oo_L^{\times}\}}.$$
Following~\cite[p. 192, Cor. 2]{bib:lang} or \cite[Thm 2]{bib:lemmer}, 
the index 
$$(\,\{a\in\Oo_L^{\times}\mid N_{L/K}(a)=1\}:\{\sigma(a)/a\mid a\in\Oo_L^{\times}\}\,)$$ is equal to the product:
$$(\Oo_K^\times:N_{L/K}(\Oo_L^\times))\times [L:K] \times\frac{1}{\prod_{\pp\vert\infty}e_{\pp}(L/K)}\,.$$ 
\end{proof}

\begin{corollary}\cite[Thm 1]{bib:lemmer}\label{th:44d}
If $L/K$ is cyclic of prime degree $p$, then $$\vert\Pp o(L/K)\vert= \frac{h_K \times p^{s_{L/K}-1}}{(\Oo_K^\times:N_{L/K}(\Oo_L^\times))}$$
where $s_{L/K}$ denotes the number of ramified primes of $K$ including those at infinity.
\end{corollary}

\noindent If $L/\SQ$ is cyclic of prime degree, we obtain \cite[Cor. 3.11]{bib:chabert2003} and \cite[Prop. 3.2]{bib:zantema}. In particular, if $L/\SQ$ is cyclic with odd prime degree:
$$\vert \Pp o(L)\vert =p^{s_L-1}$$
where $s_L$ denotes the number of prime numbers which are ramified in $L$.

\begin{remark}
Assume that $K\subseteq L$ and that $K$ and $L$ are both galoisian extensions of $\SQ$. Let $G=\Gal(L/\SQ)$ and $H=\Gal(L/K)$. Then, putting together the exact sequence ~(\ref{eq:8bis}) written for $K$ and for $L$ leads to the following commutative diagram:

\minCDarrowwidth16pt
\be \begin{CD}
1 @>>> H^1(G/H,\Oo_K^\times)  @> {\tau_K} >> \oplus_p\,\SZ/e_{p(K/\SQ)}\SZ @> >> \Pp o(K) @>>> 0 \\
@ VVV @VV\mathrm{Inf} V @VVV @ VV {\varepsilon_K^L} V @ VVV\\
1 @>>> H^1(G,\Oo_L^\times)@ >{\tau_L} >> \oplus_p\,\SZ/e_{p(L/\SQ)}\SZ @ >>> \Pp o(L)@>>> 0 \end{CD}\ee 

\bigskip

The morphisms $\tau_K$and $\tau_L$ were defined in diagramm~(\ref{CD1}) while Inf denotes the inflation morphism. The snake's lemma leads to the following exact sequence:
$$ 1\to \mathrm{Ker}(\varepsilon_K^L)\to \textrm{Coker}(\textrm{Inf}) \to \oplus_p\,\SZ/e_{p(L/K)}\SZ\to \textrm{Coker}(\varepsilon_K^L) \to 0\,.$$
If $([L:K],[K:\SQ])=1$, then $\textrm{Ker}(\varepsilon_K^L)=\{1\}$, and hence, we have the following exact sequence:
$$ 1\to \textrm{Coker}(\textrm{Inf}) \to \oplus_p\,\SZ/e_{p(L/K)}\SZ\to \Pp o(L)/\varepsilon_K^L(\Pp o(K)) \to 0\,.$$
\end{remark}

%%%%%%%%%%%%%%%%%%%%%%%%%%
\section{Non-galoisian extensions}

In the non-galoisian case, the characterization of the P\'olya fields is more difficult  and it is probably the same for the description of the P\'olya groups. Let us look at the case of a cubic number field $K$. If $K$ is a galoisian cubic number field, it follows from Corollary~\ref{th:44d} that $\vert\Pp o(K)\vert=3^{s_K-1}$, and hence, $K$ is a P\'olya field if and only if there is exactly one ramified prime in the extension $K/\SQ$. If $K$ is a non-galoisian cubic number field, it follows from \cite[Thm 1.1]{bib:zantema} that $K$ is a P\'olya field if and only if $h_K=1$, that is, the P\'olya group of $K$ is trivial if and only if the class group of $K$ is trivial. What could be the generalized assertion about P\'olya groups? Perhaps, the P\'olya group of a non-galoisian cubic number field is equal to its class group? On this line, let us recall Zantema's main result in the non-galoisian case:

\begin{theorem}\cite[Thm 6.9]{bib:zantema}\label{th:69}
Let $K$ be a number field of degree $n\geq 3$ and let $G$ be the Galois group of the normal closure of $K$ over $\SQ$. If $G$ is either isomorphic to the symmetric group $\Ss_n$ where $n\not=4$, or to the alternating group $\Aa_n$ where $n\not=3,5$, or is a Frobenius group, then the following assertions are equivalent:
\begin{enumerate}
\item every ideal $\Pi_{p^f}(K)$ where $p$ is not ramified is principal,
\item $K$ is a P\'olya field,
\item $h_K=1$.
\end{enumerate}
\end{theorem}

In other words, under the previous hypotheses, denoting by $\Pp o(K)_{nr}$ the subgroup of $\Pp o(K)$ generated by the classes of the ideals $\Pi_{p^f}(K)$ where $p$ is not ramified, one has the equivalences:
$\Pp o(K)_{nr}=\{1\}\Leftrightarrow \Pp o(K)=\{1\}\Leftrightarrow \Cc l(K)=\{1\}$.
Following the spirit if the translation of Zantema's results in Section~\ref{sec:3}, we formulate the following unreasonable conjecture:

\medskip

\noindent{\bf Conjecture}.
{\em Under the hypotheses of Theorem~\ref{th:69}, the following equalities hold:}
$$\Pp o(K)_{nr}=\Pp o(K)=\Cc l(K)\,.$$

This conjecture will be proved in a forthcoming paper~\cite{bib:chabertII}.

\medskip

\noindent{\bf Acknowledgment}.
The author wants to thank the {\em Spirou Team} (I.H.P., Paris) where a first version of this paper was discussed.

%%%%%%%%%%%%%%%%%%%
%%%%%%%%%%%%%%
%%%%%%%%%%%%%%%%%%%

\end{document}